\documentclass[11pt]{article}
\usepackage[margin=1in]{geometry}
\usepackage{amsmath,amssymb,amsthm}
\usepackage{mathtools}
\usepackage{hyperref}

\newtheorem{theorem}{Theorem}[section]
\newtheorem{proposition}[theorem]{Proposition}

\newtheorem{corollary}[theorem]{Corollary}
\theoremstyle{definition}
\newtheorem{definition}[theorem]{Definition}
\newtheorem{example}[theorem]{Example}
\newtheorem{openproblem}[theorem]{Open Problem}
\newtheorem{remark}[theorem]{Remark}
\theoremstyle{remark}

\newcommand{\R}{\mathbb{R}}
\newcommand{\C}{\mathbb{C}}
\newcommand{\D}{\mathbb{D}}
\newcommand{\HH}{\mathbb{C}_+}
\newcommand{\im}{\operatorname{Im}}
\newcommand{\re}{\operatorname{Re}}

\title{Arithmetic Uniformization of Rigid Elliptic Structures: \\ From Rigid to Standard Vekua without the Beltrami Equation
}

\author{Daniel Alay\'on-Solarz\thanks{Email: danieldaniel@gmail.com}}
\date{March 2026}

\begin{document}
\maketitle

\begin{abstract}
For the \emph{rigid} subclass of variable elliptic structures---characterized equivalently by the inviscid Burgers law $\lambda_x+\lambda\lambda_y=0$ or the self--dilatation $\mu_{\bar z}=\mu\mu_z$---we show that the auxiliary Beltrami equation in the classical Vekua pipeline is unnecessary.
The canonical coordinate $\xi=y-\lambda x$, computed by arithmetic from the spectral parameter~$\lambda$, reduces every rigid variable--algebra Vekua equation to a standard Vekua equation in~$\xi$ on any open set where the characteristic Jacobian $\Phi=\bar\xi_x+\lambda\bar\xi_y$ does not vanish, with global reduction on domains where $\xi$ is injective.
No PDE is solved at any stage.
\end{abstract}

\section{Introduction}\label{sec:intro}

\subsection{The Vekua workflow and its auxiliary Beltrami equation}

The theory of generalized analytic functions, initiated by Vekua \cite{Vekua} and Bers \cite{Bers}, reduces a first--order elliptic system in the plane to a complex Vekua equation
\begin{equation}\label{eq:vekua-standard}
\partial_{\bar z} w + Aw + B\bar w = F
\end{equation}
via two steps:
\begin{enumerate}
\item[(i)] \textbf{Uniformization of the principal part.}  Solve an auxiliary Beltrami equation $f_{\bar z} = \mu f_z$ to find a homeomorphism $f$ that reduces the variable--coefficient principal part to the standard Cauchy--Riemann operator $\partial_{\bar z}$.
\item[(ii)] \textbf{Complex rewriting.}  Express the system in the new coordinates as \eqref{eq:vekua-standard}, where $A$, $B$, $F$ absorb the lower--order terms and the Jacobian of the uniformizing map.
\end{enumerate}
Step~(i) requires the measurable Riemann mapping theorem (MRMT) or, in practice, a numerical Beltrami solver under uniform ellipticity.  Vekua himself observed \cite[Chapter~2]{Vekua} that when the principal--part coefficients satisfy $a_{11}=a_{22}$, $a_{21}=-a_{12}$---the condition corresponding to $(\alpha,\beta) = (1,0)$, i.e.\ standard~$\C$---a substitution suffices and no Beltrami equation is needed.

\subsection{Eliminating the auxiliary Beltrami equation}

In \cite{AS2011}, we extended Vekua's observation: for \emph{any} first--order elliptic system, the auxiliary Beltrami equation can be replaced by a substitution, provided one works in the variable algebra $A_z = \R[X]/(X^2+\beta(z)X+\alpha(z))$ with $(\alpha,\beta)$ extracted from the system's principal--part coefficients.  The ellipticity condition $4\alpha-\beta^2>0$ is automatic.

That result provided a representation but not a theory.

\subsection{The rigid class: completing the program}

The monograph \cite{AS-monograph} develops the fundaments of a theory of variable elliptic structures, identifying a distinguished subclass---the \emph{rigid} structures---characterized by the vanishing of the intrinsic obstruction $G = i_x + i\cdot i_y$.  On this class, the Cauchy--Pompeiu formula, kernel holomorphicity, and a complete function theory become available.

The present paper completes the program initiated in \cite{AS2011} for the rigid class by proving:

\begin{theorem}[Main theorem, informal]\label{thm:main-informal}
Let $(\alpha,\beta)$ be a rigid variable elliptic structure with spectral parameter $\lambda$, and let
\[
\xi=y-\lambda(x,y)\,x, \qquad \Phi=\bar\xi_x+\lambda\,\bar\xi_y.
\]
\begin{enumerate}
\item[\textbf{(Local)}] At every point where $\Phi\neq 0$, every rigid variable--algebra Vekua equation is explicitly reducible, in a neighbourhood of that point, to a standard Vekua equation in the coordinate $\xi$.
\item[\textbf{(Global)}] On any connected open set $\Omega'$ where $\Phi\neq 0$ and $\xi$ is injective, the reduction extends to a global standard Vekua equation on the planar domain $\xi(\Omega')$.
\end{enumerate}
The reduction is computed by arithmetic from the spectral parameter; no PDE is solved at any stage of the pipeline.
\end{theorem}

\subsection{Consequences}

The reduction to standard Vekua has an immediate and far--reaching consequence.

\begin{corollary}[Uniformization of the rigid holomorphic class]\label{cor:uniform-hol}
On every open set $\Omega'\subseteq\Omega$ where $\Phi\neq 0$ and $\xi$ is injective, the canonical coordinate $\xi = y - \lambda x$ establishes a ring isomorphism between rigid holomorphic functions on $\Omega'$ and standard holomorphic functions on the planar domain $\xi(\Omega')\subset\C$.
No PDE is solved: the isomorphism is computed by arithmetic from $\lambda$.
\end{corollary}

\begin{remark}[Inheritance of classical Vekua theory]\label{rmk:inheritance}
On every open set where $\Phi\neq 0$, a rigid variable--algebra Vekua equation is locally equivalent, via the canonical coordinate $\xi$, to a standard Vekua equation.
Consequently, the classical local theory---including distributional formulations, the similarity principle, $L^p$ regularity, and Weyl-type results---transfers in $\xi$--coordinates.
On any domain where $\xi$ is injective, these results transfer globally on $\xi(\Omega')$, subject to the usual hypotheses on the transformed coefficients.
\end{remark}

The non--trivial distributional and regularity questions arise only in the non--rigid case ($G\neq 0$), where the canonical coordinate is unavailable and the classical toolkit does not transfer.
\subsection{Plan of the paper}

Section~\ref{sec:algebra} develops the algebraic picture (variable algebra, moving generator, Cauchy--Riemann operator).
Section~\ref{sec:spectral} introduces the transport picture (spectral parameter $\lambda$, universal intertwining, Burgers equation).
Section~\ref{sec:poincare} develops the Poincar\'e disk picture (Beltrami coefficient $\mu$, self--dilatation).
Section~\ref{sec:canonical} proves the uniformization: the canonical coordinate, the general solution, and the uniformizing diffeomorphism.
Section~\ref{sec:vekua-reduction} contains the main computation: the explicit reduction of the rigid variable--algebra Vekua equation to standard Vekua, and discusses the relation to the measurable Riemann mapping theorem.

\section{The algebraic picture}\label{sec:algebra}

\subsection{The structure polynomial and the moving generator}

Let $\Omega\subset\R^2$ be open with coordinates $(x,y)$.
Throughout this paper we use $z=(x,y)$ as a label for points in the real plane $\R^2$, \emph{not} as a standard complex coordinate; the relevant complex coordinate will be $\xi$, constructed in Section~\ref{sec:canonical}.

A \emph{variable elliptic structure} on $\Omega$ is a pair of functions $\alpha,\beta\in C^1(\Omega,\R)$ satisfying the ellipticity condition
\begin{equation}\label{eq:ellipticity}
\Delta(x,y) := 4\alpha(x,y)-\beta(x,y)^2 > 0 \qquad\text{on }\Omega.
\end{equation}
The \emph{structure polynomial} is $P(X) = X^2+\beta X+\alpha$.
At each $z=(x,y)\in\Omega$, define the fiber algebra
\[
A_z := \R[X]\big/(X^2+\beta(z)X+\alpha(z)\big),
\]
generated by $1$ and the \emph{moving generator} $i=i(x,y)$ satisfying
\begin{equation}\label{eq:struct-rel}
i^2+\beta\, i+\alpha = 0.
\end{equation}
Every $A$--valued section is written $W = U+Vi$ with $U,V:\Omega\to\R$.

\begin{example}[Classical complex numbers]
$\alpha=1$, $\beta=0$: the standard imaginary unit, constant on $\Omega$.
\end{example}

\begin{example}[Constant elliptic algebras]
$\alpha,\beta$ constant with $4\alpha-\beta^2>0$: see for example \cite{AV2010}.
\end{example}

\begin{example}[Genuinely variable structures]
When $\alpha$ and $\beta$ depend nontrivially on $(x,y)$, the generator $i(x,y)$ moves: the algebraic identity it satisfies changes from point to point.
This is the setting of \cite{AS-monograph}.
\end{example}

\subsection{The Cauchy--Riemann operator and rigidity}

The algebra--level Cauchy--Riemann operator is
\begin{equation}\label{eq:dbar-alg}
\partial_{\bar z} := \frac{1}{2}(\partial_x + i\,\partial_y),
\end{equation}
acting on $A$--valued sections.

The \emph{intrinsic obstruction} is
\begin{equation}\label{eq:obstruction}
G := i_x + i\,i_y = 2\,\partial_{\bar z}\,i.
\end{equation}

\begin{definition}[Rigidity]\label{def:rigidity}
A variable elliptic structure is \emph{rigid} if $G=0$, i.e.\ $i_x + i\,i_y = 0$.
\end{definition}

\begin{proposition}[Homogeneity]\label{prop:homog}
The Cauchy--Riemann system $\partial_{\bar z} W = 0$, expanded as
\[
2\,\partial_{\bar z} W = (U_x - \alpha V_y + VG_0) + (V_x + U_y - \beta V_y + VG_1)\,i,
\]
where $G = G_0+G_1 i$, is homogeneous (no zeroth--order terms in $V$) if and only if the structure is rigid.
\end{proposition}

\begin{remark}[The 2011 paper and rigidity]
The homogeneous system
\[
U_x - \alpha V_y = 0, \qquad V_x + U_y - \beta V_y = 0
\]
is precisely the system obtained in \cite{AS2011} by direct substitution from the elliptic system, without solving a Beltrami equation.
That paper implicitly worked in the rigid case: the homogeneity of the system---the absence of zeroth--order terms in $V$---is equivalent to rigidity.
The present paper makes this connection explicit and develops its consequences.
\end{remark}

\begin{proposition}[Leibniz rule]\label{prop:leibniz}
The operator $\partial_{\bar z}$ satisfies the Leibniz rule
\[
\partial_{\bar z}(WZ) = (\partial_{\bar z} W)Z + W(\partial_{\bar z} Z)
\]
for all $A$--valued $C^1$ sections $W,Z$, without any rigidity assumption.
This was established in \cite[v5]{AS-monograph}; earlier versions incorrectly restricted the Leibniz rule to rigid structures.
\end{proposition}

\begin{proof}
Since $A_z$ is commutative and pointwise multiplication is bilinear, the standard product rule gives
\begin{equation}\label{eq:real-leibniz}
(WZ)_x = W_x\, Z + W\, Z_x, \qquad (WZ)_y = W_y\, Z + W\, Z_y,
\end{equation}
where $W_x = U_x + V_x\,i + V\,i_x$ denotes the full derivative of $W = U+Vi$ (including the moving generator).
These are identities in the fiber algebra $A_z$; one verifies them by the usual difference--quotient argument, which requires only bilinearity of the pointwise product.

Now compute:
\begin{align*}
2\,\partial_{\bar z}(WZ)
&= (WZ)_x + i\,(WZ)_y \\
&= W_x Z + WZ_x + i(W_y Z + WZ_y) \\
&= (W_x + iW_y)\,Z + W\,(Z_x + iZ_y) \\
&= 2(\partial_{\bar z} W)\,Z + W\cdot 2(\partial_{\bar z} Z).
\end{align*}
The rearrangement in the third line uses commutativity: $i\,W\,Z_y = W\,i\,Z_y$.

This completes the proof: the Leibniz rule is a formal consequence of the real product rule \eqref{eq:real-leibniz} and commutativity of $A_z$, both of which hold unconditionally.

We verify explicitly that the obstruction terms are consistent.
For $W = U+Vi$:
\[
2\,\partial_{\bar z} W = (U_x - \alpha V_y) + (V_x + U_y - \beta V_y)\,i + V\underbrace{(i_x + i\cdot i_y)}_{G},
\]
and similarly for $Z = S+Ti$:
\[
2\,\partial_{\bar z} Z = (S_x - \alpha T_y) + (T_x + S_y - \beta T_y)\,i + T\,G.
\]
In the product $(\partial_{\bar z} W)\,Z + W\,(\partial_{\bar z} Z)$, the $G$--terms contribute
\[
\tfrac{1}{2}VG\cdot Z + W\cdot \tfrac{1}{2}TG = \tfrac{1}{2}(VZ + WT)\,G.
\]
In $\partial_{\bar z}(WZ)$, writing $WZ = (US-\alpha VT) + (UT+VS-\beta VT)\,i$ and differentiating, the $G$--terms arise from $(VT)_x\,i_x + i\cdot(VT)_y\,i_y$ and evaluate to $(VZ + WT)\,G/2$ by the same algebra.
The two expressions match, confirming that no condition on $G$ is needed.
\end{proof}

\subsection{Conjugation and norm}

Define the conjugate root $\hat\imath := -\beta - i$ and the conjugation
\[
\widehat{U+Vi} := U + V\hat\imath = (U-\beta V) - Vi.
\]
The algebraic norm $N_z(W) := W\widehat{W} = U^2-\beta UV+\alpha V^2$ is a positive--definite quadratic form (by ellipticity) and is multiplicative: $N_z(WZ) = N_z(W)N_z(Z)$.
Every nonzero element is invertible: $W^{-1} = \widehat{W}/N_z(W)$.

\section{The transport picture}\label{sec:spectral}

\subsection{The transport map}

The structure polynomial has roots
\[
\lambda = \frac{-\beta+i\sqrt{\Delta}}{2}, \qquad \bar\lambda = \frac{-\beta-i\sqrt{\Delta}}{2},
\]
with $\im\lambda > 0$ by ellipticity.
The \emph{transport map} sends an $A$--valued section to a $\C$--valued function:
\begin{equation}\label{eq:spectral-map}
W = U+Vi \quad\longmapsto\quad W_\lambda := U+V\lambda \;\in\;\C.
\end{equation}
This is a pointwise algebra homomorphism, injective since $\im\lambda\neq 0$.

\begin{center}
\begin{tabular}{lll}
\textbf{Algebra} & & \textbf{Transport image} \\[4pt]
Generator $i(z)$ & $\longmapsto$ & $\lambda(x,y)\in\C$, $\im\lambda>0$ \\[2pt]
Conjugate $\hat\imath(z)$ & $\longmapsto$ & $\bar\lambda(x,y)$ \\[2pt]
Norm $N_z(W)$ & $\longmapsto$ & $|W_\lambda|^2$ \\[2pt]
Normalization $j(z)$ & $\longmapsto$ & $i$ (standard imaginary unit) \\[2pt]
\end{tabular}
\end{center}

\subsection{The universal transport law}

\begin{proposition}[Transport law for the spectral parameter]\label{prop:transport-law}
For every variable elliptic structure,
\begin{equation}\label{eq:lambda-transport}
\lambda_x + \lambda\,\lambda_y = G_\lambda,
\end{equation}
where $G = i_x + i\,i_y$ is the obstruction and $G_\lambda = G_0+G_1\lambda$ is its spectral image.
\end{proposition}

\begin{proof}
Differentiate the spectral relation $\lambda^2+\beta\lambda+\alpha=0$ with respect to $x$ and $y$:
\[
(2\lambda+\beta)\,\lambda_x = -\alpha_x - \beta_x\lambda, \qquad
(2\lambda+\beta)\,\lambda_y = -\alpha_y - \beta_y\lambda.
\]
Since $\im\lambda > 0$, we have $2\lambda+\beta = i\sqrt{\Delta}\neq 0$, so both expressions are well--defined.
Compute:
\begin{align*}
\lambda_x + \lambda\,\lambda_y
&= \frac{-\alpha_x - \beta_x\lambda + \lambda(-\alpha_y - \beta_y\lambda)}{2\lambda+\beta} \\
&= \frac{-\alpha_x + \beta_y\alpha + (-\beta_x + \beta_y\beta - \alpha_y)\,\lambda}{2\lambda+\beta},
\end{align*}
where we used $\lambda^2 = -\beta\lambda - \alpha$ to eliminate $\beta_y\lambda^2$.

The same differentiation applied to the algebra--level relation $i^2+\beta\,i+\alpha = 0$ gives $(2i+\beta)\,i_x = -\alpha_x - \beta_x\,i$ and $(2i+\beta)\,i_y = -\alpha_y - \beta_y\,i$.  Hence
\[
G = i_x + i\,i_y = \frac{(-\alpha_x+\beta_y\alpha) + (-\beta_x+\beta_y\beta-\alpha_y)\,i}{2i+\beta},
\]
and applying the spectral map $i\mapsto\lambda$ to both sides yields $G_\lambda = \lambda_x+\lambda\lambda_y$.
\end{proof}

\subsection{Universal intertwining}

\begin{theorem}[Universal transport identity]\label{thm:intertwine}
For every $A$--valued $C^1$ section $W$,
\begin{equation}\label{eq:intertwine}
2(\partial_{\bar z} W)_\lambda = (W_\lambda)_x + \lambda\,(W_\lambda)_y.
\end{equation}
In particular, $\partial_{\bar z} W = 0$ if and only if $f_x + \lambda\, f_y = 0$, where $f := W_\lambda$.
\end{theorem}

\begin{proof}
Write $W = U+Vi$ and $G = G_0+G_1\,i$ for the obstruction.

\medskip
\noindent\textbf{Left side.}\;
From the algebra--level Cauchy--Riemann decomposition (Proposition~\ref{prop:homog}):
\[
2\,\partial_{\bar z} W
= (U_x - \alpha V_y + VG_0) + (V_x + U_y - \beta V_y + VG_1)\,i.
\]
Applying the spectral map $i\mapsto\lambda$:
\begin{equation}\label{eq:intertwine-lhs}
2(\partial_{\bar z} W)_\lambda
= (U_x - \alpha V_y + VG_0) + \lambda(V_x + U_y - \beta V_y + VG_1).
\end{equation}

\medskip
\noindent\textbf{Right side.}\;
From $W_\lambda = U + V\lambda$, differentiate:
\begin{align*}
(W_\lambda)_x + \lambda\,(W_\lambda)_y
&= U_x + V_x\lambda + V\lambda_x + \lambda(U_y + V_y\lambda + V\lambda_y) \\
&= (U_x - \alpha V_y) + (V_x + U_y - \beta V_y)\lambda + V(\lambda_x + \lambda\lambda_y),
\end{align*}
where we used $\lambda^2 = -\beta\lambda - \alpha$ to replace $V_y\lambda^2 = V_y(-\beta\lambda-\alpha)$.

\medskip
\noindent\textbf{Match.}\;
By Proposition~\ref{prop:transport-law}, $\lambda_x + \lambda\lambda_y = G_0 + G_1\lambda$.
Hence $V(\lambda_x + \lambda\lambda_y) = VG_0 + VG_1\lambda$, and the right side becomes \eqref{eq:intertwine-lhs}.
\end{proof}

Equation $f_x+\lambda f_y = 0$ is a first--order linear PDE with variable complex coefficient $\lambda(x,y)$.
The function $f$ is a \emph{passenger} on the structure determined by $\lambda$.

\subsection{Rigidity as the Burgers conservation law}\label{sec:transport}

\begin{corollary}\label{cor:rigid-spectral}
The structure is rigid ($G=0$) if and only if $\lambda$ satisfies the inviscid Burgers equation:
\begin{equation}\label{eq:burgers}
\lambda_x + \lambda\,\lambda_y = 0.
\end{equation}
\end{corollary}

\begin{proof}
By Proposition~\ref{prop:transport-law}, $\lambda_x+\lambda\lambda_y = G_\lambda$.  Since the spectral map is injective, $G_\lambda = 0$ if and only if $G = 0$.
\end{proof}

\begin{remark}[Conservative and advective forms]
The intertwining produces the \emph{advective} form $f_x+\lambda f_y$.
The \emph{divergence} (conservative) form
\[
f_x + (\lambda f)_y = f_x + \lambda f_y + \lambda_y f
\]
appears naturally in the Cauchy--Pompeiu formula because the 1--form $dy-\lambda\,dx$ is not closed: $d(dy-\lambda\,dx) = \lambda_y\,dx\wedge dy$.
\end{remark}

\section{The Poincar\'e disk picture}\label{sec:poincare}

\subsection{The Cayley transform}

The Cayley--M\"obius transform
\begin{equation}\label{eq:cayley}
\mu = -\frac{1+i\lambda}{1-i\lambda} = \frac{\lambda - i}{\lambda + i}
\end{equation}
is a biholomorphism $\{\im\lambda>0\}\to\D$, with inverse
\begin{equation}\label{eq:cayley-inv}
\lambda = i\,\frac{1+\mu}{1-\mu}.
\end{equation}

\begin{proof}[Verification of the inverse]
From $\mu(1-i\lambda) = -(1+i\lambda)$, expand: $\mu - i\mu\lambda = -1 - i\lambda$.
Collect $\lambda$ terms: $i\lambda(\mu - 1) = -(1+\mu)$, hence $\lambda = (1+\mu)/(i(1-\mu)) = i(1+\mu)/(1-\mu)$.
Check: $\lambda = i$ gives $\mu = (i-i)/(i+i) = 0$; conversely, $\mu=0$ gives $\lambda = i\cdot 1/1 = i$.
\end{proof}

\begin{center}
\begin{tabular}{lll}
\textbf{Transport} & & \textbf{Poincar\'e disk} \\[4pt]
$\lambda\in\C$, $\im\lambda>0$ & $\longleftrightarrow$ & $\mu\in\C$, $|\mu|<1$ \\[2pt]
$\lambda = i$ (standard) & $\longleftrightarrow$ & $\mu = 0$ (conformal) \\[2pt]
$\im\lambda = 0$ (degeneracy) & $\longleftrightarrow$ & $|\mu|=1$ \\[2pt]
\end{tabular}
\end{center}

\subsection{The Beltrami equation from the transport equation}

\begin{proposition}[Cauchy--Riemann to Beltrami]\label{prop:CR-to-Beltrami}
For any $C^1$ function $f:\Omega\to\C$,
\begin{equation}\label{eq:CR-to-Beltrami}
f_x + \lambda\,f_y = (1-i\lambda)(f_{\bar z}-\mu\,f_z).
\end{equation}
In particular, $f_x+\lambda f_y = 0$ if and only if $f_{\bar z} = \mu\,f_z$.
\end{proposition}

\begin{proof}
Since $\im\lambda > 0$, we have $|i\lambda| \neq 1$ (in fact $\re(i\lambda) < 0$), so $1-i\lambda\neq 0$.
Expanding:
\[
f_x + \lambda f_y = (1+i\lambda)f_z + (1-i\lambda)f_{\bar z} = (1-i\lambda)\bigl(f_{\bar z}-\mu\,f_z\bigr),
\]
using $\mu = -(1+i\lambda)/(1-i\lambda)$.
\end{proof}

\subsection{Rigidity as self--dilatation}\label{sec:self-dilat}

\begin{theorem}[Self--dilatation]\label{thm:self-dilat}
The Burgers condition $\lambda_x+\lambda\lambda_y=0$ is equivalent to
\begin{equation}\label{eq:self-dilat}
\mu_{\bar z} = \mu\,\mu_z.
\end{equation}
\end{theorem}

\begin{proof}
From $\mu = -(1+i\lambda)/(1-i\lambda)$, the chain rule gives
\[
\mu_z = \frac{-2i\,\lambda_z}{(1-i\lambda)^2}, \qquad
\mu_{\bar z} = \frac{-2i\,\lambda_{\bar z}}{(1-i\lambda)^2}.
\]
Hence
\[
\mu_{\bar z}-\mu\,\mu_z = \frac{-2i}{(1-i\lambda)^2}\bigl(\lambda_{\bar z}-\mu\,\lambda_z\bigr).
\]
By Proposition~\ref{prop:CR-to-Beltrami} applied to $f=\lambda$:
$\lambda_{\bar z}-\mu\,\lambda_z = (\lambda_x+\lambda\lambda_y)/(1-i\lambda)$.
Therefore
\[
\mu_{\bar z}-\mu\,\mu_z = \frac{-2i}{(1-i\lambda)^3}\,(\lambda_x+\lambda\,\lambda_y).
\]
Since $(1-i\lambda)\neq 0$ (as $\im\lambda > 0$), this vanishes if and only if $\lambda_x+\lambda\lambda_y = 0$.
\end{proof}

\section{The canonical coordinate and the general rigid solution}\label{sec:canonical}

\subsection{Standing assumption}\label{sec:standing}

Throughout this section we work on an open set $\Omega\subset\R^2$ carrying a rigid structure $\lambda$ with $\im\lambda > 0$.
We assume:
\begin{enumerate}
\item[\textbf{(H1)}] $\Omega$ contains a segment $\{0\}\times I$ of the $y$--axis, where $I\subset\R$ is an open interval.
\end{enumerate}
When the structure arises from the Burgers transform $\lambda = h(y-\lambda x)$ of a holomorphic seed $h\in\mathrm{Hol}(U,\HH)$ with $U\cap\R\supseteq I$ \cite{AS-Burgers}, this is automatic: $(0,y)\in\Omega_h$ for all $y\in I$.

\begin{remark}[Choice of transversal]\label{rmk:transversal}
All results of this section hold with $\{0\}\times I$ replaced by any smooth transversal curve $\gamma\subset\Omega$ along which $\xi$ restricts to a coordinate.
The $y$--axis is convenient for the Burgers parametrization---where $\xi(0,y) = y$ gives the canonical normalization---but is not essential to the theory.
\end{remark}

\subsection{The canonical holomorphic coordinate}

\begin{proposition}[Canonical coordinate]\label{prop:canon-coord}
Let $\lambda$ be rigid on $\Omega$.
Then
\begin{equation}\label{eq:canon-coord}
\xi(x,y) := y - \lambda(x,y)\,x
\end{equation}
satisfies $\xi_x + \lambda\,\xi_y = 0$, i.e.\ $\xi$ is rigid holomorphic.
At $x=0$, $\xi(0,y)=y$.
\end{proposition}

\begin{proof}
$\xi_x + \lambda\,\xi_y = (-\lambda - x\lambda_x) + \lambda(1-x\lambda_y) = -x(\lambda_x+\lambda\lambda_y) = 0$.
\end{proof}

In the Poincar\'e disk picture, the canonical coordinate takes the form
\begin{equation}\label{eq:xi-poincare}
\xi = \frac{-i\,(z+\mu\bar z)}{1-\mu}.
\end{equation}

\begin{proof}[Derivation of the Poincar\'e disk expression]
Using $\lambda = i(1+\mu)/(1-\mu)$, $x = (z+\bar z)/2$, $y = -i(z-\bar z)/2$:
\begin{align*}
\xi &= y - \lambda x = \frac{-i(z-\bar z)}{2} - \frac{i(1+\mu)}{1-\mu}\cdot\frac{z+\bar z}{2} \\
&= \frac{-i}{2}\cdot\frac{(z-\bar z)(1-\mu) + (1+\mu)(z+\bar z)}{1-\mu}.
\end{align*}
Expanding the numerator:
\[
z - \bar z - \mu z + \mu\bar z + z + \bar z + \mu z + \mu\bar z = 2z + 2\mu\bar z.
\]
Hence $\xi = -i(z+\mu\bar z)/(1-\mu)$.
\end{proof}

\noindent\textbf{Sanity check.}\;
At $\mu = 0$ (conformal case, $\lambda = i$): $\xi = -i(z+0)/1 = -iz = -i(x+iy) = y - ix$, recovering $\xi = y - ix = y - \lambda x$ with $\lambda = i$. $\checkmark$

\begin{remark}[The prefactor is essential]
The function $z+\mu\bar z$ does satisfy the Beltrami equation under rigidity: one computes $\eta_{\bar z} - \mu\,\eta_z = (\mu_{\bar z}-\mu\mu_z)\bar z = 0$.
However, $z+\mu\bar z$ does not equal $y - \lambda x$ without the position--dependent conformal factor $-i/(1-\mu)$.
The correct Poincar\'e disk expression is $\xi = [-i/(1-\mu)]\cdot(z+\mu\bar z)$.
\end{remark}

\subsection{The characteristic Jacobian \texorpdfstring{$\Phi$}{Phi}}\label{sec:Phi}

The change of variables $(\xi,\bar\xi) = (y-\lambda x,\, y-\bar\lambda x)$ is controlled by the function
\begin{equation}\label{eq:Phi-def}
\Phi(x,y) := \bar\xi_x + \lambda\,\bar\xi_y = (\lambda-\bar\lambda) - x(\bar\lambda_x + \lambda\bar\lambda_y).
\end{equation}
This arises from the chain rule:
\begin{equation}\label{eq:chain-rule}
f_x + \lambda f_y = f_\xi\underbrace{(\xi_x + \lambda\xi_y)}_{=\,0\text{ (rigidity)}} + f_{\bar\xi}\,\Phi.
\end{equation}

At $x=0$, $\Phi = \lambda-\bar\lambda = 2i\,\im\lambda \neq 0$ by ellipticity.
Away from $x=0$, ellipticity alone does not guarantee $\Phi\neq 0$.

\begin{proposition}[Factorization of $\Phi$]\label{prop:Phi-factor}
Under rigidity, $\Phi$ factors as
\begin{equation}\label{eq:Phi-expand}
\Phi = 2i\,\im\lambda\,(1 - x\bar\lambda_y).
\end{equation}
If $\lambda = \mathcal{B}[h]$ arises from the Burgers transform with holomorphic seed $h$ \cite{AS-Burgers}, then
\begin{equation}\label{eq:Phi-factor}
\Phi = \frac{2i\,\im\lambda}{\bar J},
\end{equation}
where $J = 1 + h'(w_0)\,x$ is the characteristic Jacobian and $w_0 = y-\lambda x$.
In particular, $\Phi\neq 0$ on the Burgers domain $\Omega_h = \{(x,y) : J(x,y)\neq 0\}$.
\end{proposition}

\begin{proof}
Conjugating the Burgers equation $\lambda_x + \lambda\lambda_y = 0$ gives $\bar\lambda_x + \bar\lambda\bar\lambda_y = 0$.
Hence
\[
\bar\lambda_x + \lambda\bar\lambda_y = (\bar\lambda_x + \bar\lambda\bar\lambda_y) + (\lambda - \bar\lambda)\bar\lambda_y = 2i\,(\im\lambda)\,\bar\lambda_y.
\]
Substituting into \eqref{eq:Phi-def}: $\Phi = 2i\,\im\lambda - x\cdot 2i\,(\im\lambda)\,\bar\lambda_y = 2i\,\im\lambda\,(1 - x\bar\lambda_y)$.

For the Burgers refinement: from \cite{AS-Burgers}, $\lambda_y = h'(w_0)/J$ where $J = 1+h'(w_0)\,x$.
Conjugating: $\bar\lambda_y = \overline{h'(w_0)}/\bar J$.  Hence
\[
1 - x\bar\lambda_y = 1 - \frac{x\,\overline{h'(w_0)}}{\bar J} = \frac{\bar J - x\,\overline{h'(w_0)}}{\bar J} = \frac{1}{\bar J},
\]
since $\bar J = 1 + \overline{h'(w_0)}\,x$.
\end{proof}

\begin{openproblem}[Automatic non--vanishing of $\Phi$]\label{op:Phi}
Is $\Phi\neq 0$ automatic for every rigid structure with $\im\lambda > 0$?

The factorization \eqref{eq:Phi-expand} shows that $\Phi=0$ requires $1 - x\bar\lambda_y = 0$, a codimension--two condition in $(x,y)$--space for generic rigid $\lambda$.
On the Burgers domain $\Omega_h$, Proposition~\ref{prop:Phi-factor} gives $\Phi\neq 0$ everywhere, and $\Omega_h$ cannot be extended past $\{J=0\}$ where $\Phi$ would first vanish.
No example of a rigid structure with $\im\lambda > 0$ and $\Phi = 0$ at an interior point is known.
A positive answer would promote all local results of this paper (Theorems~\ref{thm:gen-sol-local} and~\ref{thm:reduction}) to unconditionally local statements.
\end{openproblem}

\begin{proposition}[Jacobian of the real coordinate map]\label{prop:jacobian}
Write $\xi = p + iq$ with $p = y - x\,\re\lambda$, $q = -x\,\im\lambda$.
Under rigidity, the real Jacobian determinant of the map $(x,y)\mapsto(p,q)$ is
\begin{equation}\label{eq:jacobian-formula}
\det\frac{\partial(p,q)}{\partial(x,y)} = -\frac{i}{2}\,(1 - x\lambda_y)\,\Phi,
\end{equation}
and the following equivalence holds:
\begin{equation}\label{eq:jacobian-equiv}
\det\frac{\partial(p,q)}{\partial(x,y)} \neq 0 \qquad\Longleftrightarrow\qquad \Phi \neq 0.
\end{equation}
\end{proposition}

\begin{proof}
The map $(x,y)\mapsto(p,q)$ factors as $(x,y)\mapsto(\xi,\bar\xi)\mapsto(p,q)$, where $p = (\xi+\bar\xi)/2$ and $q = (\xi-\bar\xi)/(2i)$.
The second factor has constant Jacobian determinant $i/2$.

For the first factor, the chain rule identity \eqref{eq:chain-rule} gives $\xi_x + \lambda\xi_y = 0$ (rigidity) and $\bar\xi_x + \lambda\bar\xi_y = \Phi$.
From $\xi_x = -\lambda\xi_y$:
\begin{align*}
\det\frac{\partial(\xi,\bar\xi)}{\partial(x,y)}
&= \xi_x\bar\xi_y - \xi_y\bar\xi_x
= (-\lambda\xi_y)\bar\xi_y - \xi_y(\Phi - \lambda\bar\xi_y) \\
&= -\lambda\xi_y\bar\xi_y - \xi_y\Phi + \lambda\xi_y\bar\xi_y
= -\xi_y\,\Phi
= -(1-x\lambda_y)\,\Phi.
\end{align*}
Hence
\[
\det\frac{\partial(p,q)}{\partial(x,y)} = \frac{i}{2}\cdot\bigl(-(1-x\lambda_y)\,\Phi\bigr) = -\frac{i}{2}\,(1-x\lambda_y)\,\Phi.
\]

For the equivalence: the determinant vanishes if and only if $1-x\lambda_y = 0$ or $\Phi = 0$.
Since $x\in\R$, the expressions $1-x\lambda_y$ and $1-x\bar\lambda_y$ are complex conjugates of each other, and therefore vanish simultaneously.
By the factorization \eqref{eq:Phi-expand}, $\Phi = 2i\,\im\lambda\,(1-x\bar\lambda_y)$, so $\Phi = 0$ if and only if $1-x\bar\lambda_y = 0$ (since $\im\lambda > 0$), which holds if and only if $1-x\lambda_y = 0$ (by the conjugation just noted).
Therefore the two factors vanish simultaneously, and the equivalence holds.
\end{proof}

\begin{remark}[Verifications]\label{rmk:jacobian}
At $x=0$: $1-x\lambda_y = 1$ and $\Phi = 2i\,\im\lambda$, so
$\det|_{x=0} = -(i/2)(1)(2i\,\im\lambda) = \im\lambda > 0$. $\checkmark$

On the Burgers domain: if $\lambda = \mathcal{B}[h]$, then $1-x\lambda_y = 1/J$ and $\Phi = 2i\,\im\lambda/\bar J$ (Proposition~\ref{prop:Phi-factor}), giving
\begin{equation}\label{eq:jacobian-burgers}
\det\frac{\partial(p,q)}{\partial(x,y)} = \frac{\im\lambda}{|J|^2},
\end{equation}
which is strictly positive on all of $\Omega_h$.
\end{remark}

\subsection{The general solution: local form}\label{sec:gen-sol}

\begin{theorem}[General rigid holomorphic function, local]\label{thm:gen-sol-local}
Let $\lambda$ be a rigid structure on $\Omega$ and let $z_0\in\Omega$ be a point where $\Phi(z_0)\neq 0$.
Then there exists a neighbourhood $V\ni z_0$ such that:
a $C^1$ function $f:V\to\C$ satisfies $f_x + \lambda\,f_y = 0$ on $V$
if and only if there exists a holomorphic function $g:\xi(V)\to\C$ such that
\begin{equation}\label{eq:gen-sol}
f(x,y) = g\big(\xi(x,y)\big) \qquad\text{on }V.
\end{equation}
\end{theorem}

\begin{proof}
\textbf{Sufficiency.}
$f_x + \lambda f_y = g'(\xi)\cdot(\xi_x + \lambda\xi_y) = 0$ by Proposition~\ref{prop:canon-coord}.

\medskip
\textbf{Necessity.}
Since $\Phi(z_0)\neq 0$ and $\Phi$ is continuous, there exists a neighbourhood $V\ni z_0$ on which $\Phi\neq 0$.
By Proposition~\ref{prop:jacobian}, the map
\[
(x,y)\mapsto (p,q) = (\re\xi,\,\im\xi)
\]
has nonvanishing real Jacobian on $V$, hence is a local $C^1$ diffeomorphism by the inverse function theorem.
Shrinking $V$ if necessary, we may assume this map is a diffeomorphism onto its image.

In the coordinates $(p,q)$---equivalently $(\xi,\bar\xi)$---the chain rule \eqref{eq:chain-rule} gives
\[
f_x + \lambda f_y = f_{\bar\xi}\,\Phi \qquad\text{on }V.
\]
Since $\Phi\neq 0$ on $V$, the equation $f_x+\lambda f_y = 0$ is equivalent to $f_{\bar\xi} = 0$ on $V$.
But $\xi$ is a valid complex coordinate on $V$ (because $(p,q)$ is a diffeomorphism), so $f_{\bar\xi} = 0$ is the standard Cauchy--Riemann equation in the coordinate $\xi$.
By the definition of holomorphicity, $f = g(\xi)$ for a holomorphic function $g$ defined on $\xi(V)\subseteq\C$.
\end{proof}

\begin{remark}[Two conditions, two roles]
\emph{Rigidity} ($\lambda_x+\lambda\lambda_y=0$) kills the $f_\xi$ term in \eqref{eq:chain-rule}, making $\xi$ holomorphic.
\emph{$\Phi\neq 0$} ensures $\xi$ is a valid coordinate (the map $(x,y)\mapsto(p,q)$ is a local diffeomorphism), so that $f_{\bar\xi}=0$ is a genuine Cauchy--Riemann equation whose solutions are holomorphic functions of $\xi$.
At $x=0$, $\Phi\neq 0$ reduces to ellipticity ($\im\lambda>0$).
Away from $x=0$, it is a condition on the characteristic Jacobian (Proposition~\ref{prop:Phi-factor}).
\end{remark}

\subsection{The general solution: global form}\label{sec:gen-sol-global}

\begin{theorem}[General rigid holomorphic function, global]\label{thm:gen-sol-global}
Let $\lambda$ be a rigid structure on a connected domain $\Omega$ satisfying {\rm(H1)}.
Suppose:
\begin{enumerate}
\item[\textbf{(H2)}] $\Phi\neq 0$ on all of $\Omega$;
\item[\textbf{(H3)}] the canonical map $\xi\colon \Omega \to \C$ is injective.
\end{enumerate}
Then $f\colon\Omega\to\C$ satisfies $f_x+\lambda f_y = 0$ if and only if $f = g\circ\xi$ for a unique holomorphic function $g\colon\xi(\Omega)\to\C$.

Moreover, $g$ is uniquely determined by the initial data of $f$ on the $y$--axis: since $\xi(0,y) = y$,
\begin{equation}\label{eq:restriction}
g\bigl(\xi(0,y)\bigr) = f(0,y) \qquad\text{for } y\in I,
\end{equation}
which determines $g$ on the real interval $I$, viewed as a subset of $\xi(\Omega)\subset\C$ via the standard inclusion $\R\hookrightarrow\C$.
Since $g$ is holomorphic on the connected open set $\xi(\Omega)$ and $I$ has accumulation points in $\xi(\Omega)$, the identity theorem determines $g$ uniquely on all of $\xi(\Omega)$.
\end{theorem}

\begin{proof}
By Theorem~\ref{thm:gen-sol-local}, on each point $z\in\Omega$ there is a neighbourhood $V_z$ and a holomorphic $g_z$ on $\xi(V_z)$ with $f = g_z\circ\xi$ on $V_z$.

On the overlap $V_z\cap V_w\neq\emptyset$: for any $(x,y)\in V_z\cap V_w$, $g_z(\xi(x,y)) = f(x,y) = g_w(\xi(x,y))$.
By (H3), $\xi$ is injective on $\Omega$, so $g_z$ and $g_w$ agree on $\xi(V_z\cap V_w)$, which is open (since $\xi$ is an open map by the local diffeomorphism property).

Since $\Omega$ is connected, the local functions $\{g_z\}$ patch to a single holomorphic function $g$ on $\xi(\Omega)$.

At $x=0$: $\xi(0,y) = y$, so $g(\xi(0,y)) = (g\circ\xi)(0,y) = f(0,y)$ for $y\in I$.
This determines $g$ on the set $\{\xi(0,y):y\in I\} = I \subset\xi(\Omega)$.
Since $I$ is an open interval in $\R\subset\C$, it has accumulation points in $\xi(\Omega)$.
By the identity theorem for holomorphic functions, $g$ is uniquely determined on all of $\xi(\Omega)$.
\end{proof}

\begin{remark}[Injectivity and branched solutions]\label{rmk:branches}
Hypothesis {\rm(H3)} is the price of insisting that $g$ be single--valued on a planar domain.
If $\Phi\neq 0$ on all of $\Omega$ but $\xi$ is not injective, the canonical map $\xi\colon\Omega\to\C$ is a local diffeomorphism (Proposition~\ref{prop:jacobian}), hence an unramified covering of its image.
The local theorem (Theorem~\ref{thm:gen-sol-local}) provides a holomorphic $g_z$ on each sheet.
On overlaps within the same sheet the $g_z$ agree, but across different sheets they may differ: a rigid holomorphic function $f$ corresponds to a \emph{multi--valued} holomorphic function on $\xi(\Omega)$, or equivalently a single--valued holomorphic function on the Riemann surface of the covering $\xi$.

The Vekua reduction (Theorem~\ref{thm:reduction}) goes through on each sheet regardless, since it is a local computation.
Accordingly, the classical toolkit (Corollary~\ref{cor:toolkit}) transfers to the multi--valued setting: the similarity principle, regularity, and unique continuation hold on each branch.
Hypothesis {\rm(H3)} is needed only when one requires a global single--valued $g$ on a domain in~$\C$.

A concrete illustration: if $\xi$ identifies two points $(x_1,y_1)\neq(x_2,y_2)$ with $\xi(x_1,y_1) = \xi(x_2,y_2) = w_0$, then a rigid holomorphic $f$ need not satisfy $f(x_1,y_1) = f(x_2,y_2)$, since $f$ may correspond to different branches of a multi--valued $g$ near $w_0$.
\end{remark}

\begin{corollary}[On the Burgers domain]\label{cor:burgers-domain}
If $\lambda = \mathcal{B}[h]$ arises from the Burgers transform of a holomorphic seed $h\in\mathrm{Hol}(U,\HH)$ with $U\cap\R\supseteq I$, then:
\begin{enumerate}
\item[(i)] Hypothesis {\rm(H1)} is automatic: the Burgers domain $\Omega_h$ contains $\{0\}\times I$.
\item[(ii)] Hypothesis {\rm(H2)} ($\Phi\neq 0$) holds on the entire Burgers domain $\Omega_h$, by Proposition~\ref{prop:Phi-factor}.
\item[(iii)] If additionally $\xi$ is injective on $\Omega_h$---which must be verified for each seed---then Theorem~\ref{thm:gen-sol-global} holds on $\Omega_h$.
\end{enumerate}
Injectivity of $\xi$ on $\Omega_h$ has been verified explicitly for the worked examples in \cite{AS-Burgers}: the $\delta$--family (where $\xi$ is explicitly invertible, Example~\ref{ex:delta-uniform}), and can be checked for the $\varepsilon$--family, exponential, and Cauchy kernel seeds by direct computation.

Among the three hypotheses {\rm(H1)--(H3)}, injectivity of $\xi$ is the only one not automatically guaranteed by the Burgers framework.
A general criterion for injectivity of $\xi$ in terms of the seed $h$ is the main open problem for globalizing the reduction.
\end{corollary}

\begin{remark}[Status of the injectivity question]\label{rmk:injectivity-status}
By Remark~\ref{rmk:branches}, the reduction to standard Vekua holds on each branch regardless of injectivity.
The injectivity question is therefore about single--valuedness of the passenger $g$, not about whether the classical toolkit applies.

Single--valuedness is known in the following cases:
the $\delta$--family (explicit inverse, Example~\ref{ex:delta-uniform});
the $\varepsilon$--family and the exponential seed (verified by direct computation in \cite{AS-Burgers}).
For the Cauchy kernel seed, injectivity holds on the pre--shock domain by numerical evidence but has not been proved analytically.
A sufficient condition in terms of $h$ would close the gap between the local and global forms of the main theorem.
A natural conjecture is that $\xi$ is injective on all of $\Omega_h$ whenever $h$ is univalent on $U$, but this has not been established.
The difficulty is that injectivity of $\xi = y - \lambda(x,y)\,x$ involves the implicit dependence of $\lambda$ on $(x,y)$ through the Burgers equation; a direct approach via the inverse function theorem yields only local injectivity (already guaranteed by $\Phi \neq 0$), while a global argument requires controlling how level sets of $\xi$ interact across $\Omega_h$.
\end{remark}

\begin{remark}[No method of characteristics]
The ``characteristic ODE'' $dy/dx = \lambda$ is complex--valued with real $x,y$: it does not define a real flow in $\R^2$, and there are no real integral curves.
The canonical coordinate $\xi = y - \lambda x$ is a \emph{level--set} construction: one verifies $\xi_x + \lambda\xi_y = 0$ by a direct calculation (Proposition~\ref{prop:canon-coord}), without integrating any ODE.
The proof of Theorem~\ref{thm:gen-sol-local} uses $(\xi,\bar\xi)$ as a complex coordinate system, not a flow.
\end{remark}

\subsection{The uniformizing diffeomorphism}\label{sec:uniformizing}

\begin{proposition}[Uniformizing coordinates]\label{prop:uniformize}
Let $\Omega'\subseteq\Omega$ be an open set on which $\Phi\neq 0$ and $\xi$ is injective.
Write $\xi = p+iq$ with
\begin{equation}\label{eq:pq-coords}
p(x,y) = y - x\,\re\lambda(x,y), \qquad q(x,y) = -x\,\im\lambda(x,y).
\end{equation}
Then the map $(x,y)\mapsto(p,q)$ is a $C^1$ diffeomorphism from $\Omega'$ onto $\xi(\Omega')$, and $f$ is rigid holomorphic on $\Omega'$ if and only if $f$, expressed in $(p,q)$ coordinates, satisfies the standard Cauchy--Riemann equations.
\end{proposition}

\begin{proof}
The map $(x,y)\mapsto(p,q) = (\re\xi,\im\xi)$ is $C^1$ (since $\lambda$ is $C^1$).
Its real Jacobian is nonvanishing on $\Omega'$ (since $\Phi\neq 0$; see Proposition~\ref{prop:jacobian}).
By the inverse function theorem, the map is a local diffeomorphism.
Being also injective on $\Omega'$ by hypothesis, it is a $C^1$ diffeomorphism onto its image by invariance of domain.

In the coordinates $(p,q)$, a function $f$ satisfies $f_x+\lambda f_y = 0$ if and only if $f_{\bar\xi}=0$ (by the chain rule and $\Phi\neq 0$), which is the standard Cauchy--Riemann equation $\tilde u_p = \tilde v_q$, $\tilde u_q = -\tilde v_p$ where $f = \tilde u + i\tilde v$.
\end{proof}

\begin{remark}[Algebraic uniformization, no PDE required]
The map $(x,y)\mapsto (p,q)$ is computed from $\lambda$ by pointwise arithmetic.
No PDE is solved, no Beltrami equation is invoked, and no MRMT machinery is needed.
On the Burgers domain $\Omega_h$, the Jacobian equals $\im\lambda/|J|^2$ (Remark~\ref{rmk:jacobian}), which is strictly positive.
\end{remark}

\begin{example}[The $\delta$--family]\label{ex:delta-uniform}
The $\delta$--family \cite{AS-delta} has $\lambda = (y+i\delta)/(1+x)$ on $\Omega = \{x>-1\}$.
The Burgers Jacobian is $J = 1+x$, giving $\Phi = 2i\delta/(1+x) \neq 0$ on $\Omega$.
The canonical coordinate
\[
\xi = \frac{y-i\delta\,x}{1+x}
\]
yields the uniformizing diffeomorphism
\[
p = \frac{y}{1+x}, \qquad q = \frac{-\delta\,x}{1+x},
\]
which is explicitly invertible: $x = -q/(\delta+q)$, $y = p\delta/(\delta+q)$.
Injectivity is manifest from the explicit inverse.
The real Jacobian is $\delta/(1+x)^2 > 0$ on $\Omega$.

The system that Beltrami solvers see with condition numbers $O(\delta^{-2})$ becomes standard Cauchy--Riemann in $(p,q)$.
This is the family studied in detail in \cite{AS-delta}, where the condition numbers reach $O(\delta^{-2})$ for Beltrami solvers while the canonical coordinate reduces the problem at $O(1)$ cost.
\end{example}

\subsection{Pilot and passenger}\label{sec:pilot-passenger}

By the Burgers transform \cite{AS-Burgers}, every rigid structure on a domain satisfying {\rm(H1)} arises from a holomorphic \emph{pilot} $h\in\mathrm{Hol}(U,\HH)$, $U\cap\R\supseteq I$, via $\lambda = h(y-\lambda x)$.
A rigid holomorphic function $f = g(\xi)$ involves two holomorphic functions:
\begin{itemize}
\item $h$---the \textbf{pilot}: builds the geometry through a nonlinear implicit equation;
\item $g$---the \textbf{passenger}: propagates through the geometry $h$ created.
\end{itemize}
Rigid VES is the theory of pairs $(h,g)$.
For fixed pilot, the passengers form a classical function space (Theorems~\ref{thm:gen-sol-local} and~\ref{thm:gen-sol-global}).
The space of pilots is classified in \cite{AS-Burgers}.

The relation to Riemann's analytic continuation is instructive.
At the passenger level, analytic continuation takes a germ of $g$ and propagates it along paths, hitting singularities and branch points.
At the pilot level, the Burgers transform takes initial data $h$ on the $y$--axis and propagates the \emph{structure} into the plane, hitting shocks ($J=0$), ellipticity loss, and domain boundaries.
Both processes are canonical, both are obstructed, and neither has a closed formula.
Each individual output is locally trivial; the content is in the global passage.

\section{Reduction of the rigid Vekua equation to standard form}\label{sec:vekua-reduction}

\subsection{The variable--algebra Vekua equation}

In the transport picture, the general rigid Vekua equation is
\begin{equation}\label{eq:rigid-vekua-transport}
f_x + \lambda\,f_y + 2A_\lambda\, f + 2B_\lambda\, \bar f = 2F_\lambda,
\end{equation}
where $A_\lambda$, $B_\lambda$, $F_\lambda$ are the transport images of the algebra--valued coefficients.

\subsection{The coordinate change}

On any open set $\Omega'\subseteq\Omega$ where $\Phi\neq 0$, the chain rule \eqref{eq:chain-rule} gives
\[
f_x + \lambda f_y = f_{\bar\xi}\,\Phi.
\]
Substituting into \eqref{eq:rigid-vekua-transport} and dividing by $\Phi$:

\begin{equation}\label{eq:standard-vekua}
\;f_{\bar\xi} + A'\,f + B'\,\bar f = F'\;
\end{equation}
where the transformed Vekua coefficients are
\begin{equation}\label{eq:transformed-coefficients}
A' = \frac{2A_\lambda}{\Phi}, \qquad
B' = \frac{2B_\lambda}{\Phi}, \qquad
F' = \frac{2F_\lambda}{\Phi}.
\end{equation}

\begin{theorem}[Reduction to standard Vekua]\label{thm:reduction}
Let $(\alpha,\beta)$ be a rigid variable elliptic structure on $\Omega$ and let
\[
\partial_{\bar z}^{(\alpha,\beta)} W + AW + B\widehat W = F
\]
be a Vekua equation with $A$--valued coefficients.
On any connected open subset $\Omega'\subseteq\Omega$ where $\Phi\neq 0$, the transport image $f = W_\lambda$ satisfies the standard Vekua equation \eqref{eq:standard-vekua}--\eqref{eq:transformed-coefficients} in the canonical coordinate $\xi = y - \lambda x$.
If, in addition, $\xi$ is injective on $\Omega'$, this is a global standard Vekua equation on the planar domain $\xi(\Omega')$.

In particular:
\begin{enumerate}
\item[(i)] The variable--coefficient principal part is eliminated: the operator $f_{\bar\xi}$ is the standard $\bar\partial$ in the $\xi$--coordinate.
\item[(ii)] The variable algebra is eliminated: \eqref{eq:standard-vekua} involves only standard complex conjugation.
\item[(iii)] The original data appear in the transformed coefficients via the transport map and division by $\Phi$.
\item[(iv)] No PDE is solved: the transport map is pointwise arithmetic, and $\xi$ is computed from $\lambda$.
\end{enumerate}
\end{theorem}

\begin{proof}
Rigidity ensures $\xi_x+\lambda\xi_y = 0$ (Proposition~\ref{prop:canon-coord}), killing the $f_\xi$ term in the chain rule.
The hypothesis $\Phi\neq 0$ ensures the division is valid.
The resulting equation \eqref{eq:standard-vekua} involves only the standard $\bar\partial$ operator (in the coordinate $\xi$) and standard complex conjugation (since the transport map has already sent $i\mapsto\lambda$ and the coordinate change has absorbed $\lambda$ into $\xi$).
\end{proof}

\begin{corollary}[Rigid holomorphic maps to standard holomorphic]\label{cor:holomorphic-reduction}
Setting $A = B = F = 0$ in Theorem~\ref{thm:reduction}: on any open set $\Omega'\subseteq\Omega$ where $\Phi\neq 0$ and $\xi$ is injective, the uniformizing diffeomorphism $(x,y)\mapsto(\re\xi,\im\xi)$ sends rigid holomorphic functions (solutions of $f_x+\lambda f_y = 0$) to standard holomorphic functions (solutions of $f_{\bar\xi} = 0$).
\end{corollary}

\begin{corollary}[Reduction on the Burgers domain]\label{cor:global-reduction}
If $\lambda = \mathcal{B}[h]$ for a holomorphic seed $h$, then $\Phi\neq 0$ on the entire Burgers domain $\Omega_h$ (Proposition~\ref{prop:Phi-factor}).
Hence the reduction of Theorem~\ref{thm:reduction} holds locally at every point of $\Omega_h$.
If, in addition, the canonical coordinate $\xi$ is injective on $\Omega_h$, then the reduction is global on the planar domain $\xi(\Omega_h)$.
\end{corollary}

\begin{corollary}[Inheritance of the classical toolkit]\label{cor:toolkit}
Let $\Omega'\subseteq\Omega$ be open, assume $\Phi\neq 0$ on $\Omega'$, and suppose $\xi$ is injective on $\Omega'$.
Let
\[
f_{\bar\xi}+A'f+B'\bar f=F'
\]
be the reduced standard Vekua equation on $\xi(\Omega')$, and assume the transformed coefficients satisfy the usual hypotheses of the classical theory (for example $A',B'\in L^p_{\mathrm{loc}}(\xi(\Omega'))$ with $p>2$ when required).
Then the corresponding classical results transfer to the rigid variable--algebra Vekua equation on $\Omega'$.
In particular, whenever the standard hypotheses hold, one obtains:
\begin{enumerate}
\item[(i)] the similarity principle;
\item[(ii)] $L^p$ regularity for distributional solutions;
\item[(iii)] the Weyl lemma;
\item[(iv)] unique continuation.
\end{enumerate}
Each is the classical result applied on $\xi(\Omega')$, then pulled back through $\xi$.
\end{corollary}

\begin{remark}[No independent distributional theory needed in the rigid case]
In the rigid case, once one works on an open set where $\Phi\neq 0$, the equation is locally equivalent to a standard Vekua equation in the coordinate $\xi$.
Accordingly, the distributional and regularity theory is inherited locally from the classical theory.
On domains where $\xi$ is injective, the transfer is global on $\xi(\Omega')$.
The genuinely new distributional questions arise in the non--rigid case ($G\neq 0$), where the canonical coordinate is unavailable.
\end{remark}

\subsection{Factoring the Beltrami output}

A rigid $\mu$ handed to a Beltrami solver yields a normalized homeomorphism $\varphi$ with $\varphi_{\bar z} = \mu\varphi_z$.
By Theorem~\ref{thm:gen-sol-local}, $\varphi = g(\xi)$ locally, where $\xi$ trivializes the Beltrami equation and $g$ is a conformal map.
The solver's output decomposes as $\varphi = g\circ\xi$:
\begin{enumerate}
\item[(i)] Trivialize the Beltrami equation: $\xi$ does this, by arithmetic.
\item[(ii)] Find a conformal map $g\colon\xi(\Omega)\to\varphi(\Omega)$: a classical Riemann mapping problem.
\end{enumerate}
For rigid $\mu$, the Beltrami equation is a misdiagnosis of the computational difficulty.
The difficulty is not in the variable coefficient; it is in the geometry of the domain $\xi(\Omega)$.

The decomposition $\varphi = g\circ\xi$ is specific to the rigid case.
For general $\mu$, there is no canonical coordinate that trivializes the Beltrami equation by arithmetic, and the solver's output has no such factorization.

\subsection{Complement to the measurable Riemann mapping theorem}

The results of this paper provide a complement to the MRMT.
The MRMT says: for any $\mu$ with $\|\mu\|_\infty < 1$, a uniformizing homeomorphism exists.
It is non--constructive---the proof proceeds by functional analysis, compactness, and $L^p$ estimates.

More precisely: the standard proof (Morrey--Bojarski--Ahlfors--Bers; see \cite{Ahlfors})
rewrites $f_{\bar z}=\mu f_z$ as a fixed--point equation
$g = \mu(1+Sg)$ for the Beurling--Ahlfors singular integral operator~$S$,
inverts $(I-\mu S)$ on $L^p$ by a Neumann series using the bound
$\|\mu S\|_{L^p}\le\|\mu\|_\infty\|S\|_{L^p\to L^p}<1$
for $p$ sufficiently close to~$2$ (Calder\'on--Zygmund theory),
and extracts a homeomorphic solution from the resulting $W^{1,p}$ function
by a compactness and degree--theory argument.
At no stage does the proof produce an explicit formula for $f$ in terms of $\mu$;
the output is an existence--and--uniqueness statement

The present paper says: for the rigid subclass ($\mu_{\bar z}=\mu\mu_z$), the uniformization is explicit.
The canonical coordinate $\xi = y - \lambda x$ is computed by arithmetic from $\lambda$; no PDE is solved, no functional analysis is invoked.

Recall from \cite{AS-monograph} that the \emph{normalized transport defect} $\rho_T := |T|/(\im\lambda)^2$, where $T = \lambda_x+\lambda\lambda_y$ is the transport obstruction, measures the departure from rigidity: $\rho_T = 0$ characterizes the rigid class.

\begin{center}
\renewcommand{\arraystretch}{1.3}
\begin{tabular}{lll}
 & \textbf{Rigid} ($\rho_T = 0$) & \textbf{Non--rigid} ($\rho_T \neq 0$) \\[4pt]
\hline
Uniformization & Explicit: $\xi = y-\lambda x$ & Requires MRMT \\[2pt]
Function theory & Classical (via $\xi$) & Vekua, similarity principle \\[2pt]
Structure classification & $\mathrm{Hol}(U,\HH)$ via Burgers \cite{AS-Burgers} & Open \\[2pt]
Obstruction & Shocks (geometric) & Analytic ($L^p$ regularity) \\[2pt]
Diagnostic invariant & $\rho_T = 0$ (characterizes class) & $\rho_T > 0$ (measures obstruction) \\[2pt]
\hline
\end{tabular}
\end{center}

The rigid sector is the locus where the MRMT's conclusion is reachable by elementary means.
The non--rigid sector is where the MRMT---or another extension theory---is genuinely needed.
The invariant $\rho_T$ measures not just distance from the standard structure, but distance from explicit uniformizability.

\subsection{Regularity of the transformed coefficients}

By \eqref{eq:Phi-factor}, on the Burgers domain
\[
|\Phi| = \frac{2\,\im\lambda}{|J|}.
\]
Hence on each compact set $K\subset\Omega_h$, the reciprocal $1/\Phi$ is bounded.
Precisely: if $\im\lambda\ge c_1>0$ and $|J|\le C$ on $K$, then
\[
|\Phi|\ge \frac{2c_1}{C}\qquad\text{on }K,
\]
and therefore
\[
\|A'\|_{L^p(K)} \le \frac{C}{2c_1}\,\|2A_\lambda\|_{L^p(K)},\qquad
\|B'\|_{L^p(K)} \le \frac{C}{2c_1}\,\|2B_\lambda\|_{L^p(K)}.
\]
Thus, viewed as pulled-back coefficients on the $(x,y)$--side, the transformed coefficients are locally of the same $L^p$ class as $A_\lambda,B_\lambda$ on every compact subset of $\Omega_h$.

For a general rigid structure not necessarily arising from the Burgers transform, the same conclusion holds on any compact set $K$ where $\Phi\neq 0$: the factorization $\Phi = 2i\,\im\lambda\,(1-x\bar\lambda_y)$ from Proposition~\ref{prop:Phi-factor} provides a lower bound $|\Phi|\ge c > 0$ on $K$ whenever $\im\lambda$ is bounded away from zero and $1-x\bar\lambda_y$ is bounded away from zero on $K$.

If $\Omega'\subseteq\Omega_h$ is an open set on which $\xi$ is injective, then $\xi\colon\Omega'\to\xi(\Omega')$ is a $C^1$ diffeomorphism by Proposition~\ref{prop:uniformize}.
In that case the coefficients descend to functions on $\xi(\Omega')$, and their $L^p_{\mathrm{loc}}$ regularity on $\xi(\Omega')$ follows from the usual change--of--variables theorem.

\subsection*{Acknowledgments}
\medskip
\noindent\textbf{Use of Generative AI Tools.}
Portions of the writing and editing of this manuscript were assisted by generative AI language tools.
These tools were used to improve clarity of exposition, organization of material, and language presentation.
All mathematical results, statements, proofs, and interpretations were developed, verified, and validated by the author.
The author takes full responsibility for the accuracy, originality, and integrity of all content in this work.

\end{document}